\documentclass[11pt]{article}

\usepackage[margin=1.3in]{geometry}

\usepackage[dvips]{graphics}
\usepackage{epsfig}
\usepackage{multirow}
\usepackage{array}
\usepackage{extarrows}
\usepackage{graphicx}

\usepackage{comment}
\usepackage[titletoc,title]{appendix}
\numberwithin{equation}{section}
\usepackage{amsthm}
\usepackage{enumerate}
\usepackage{amsmath}

\usepackage{amssymb,amsmath}
\usepackage{latexsym}
\usepackage[usenames]{color}
\usepackage{pgf}

\usepackage{hyperref}
\hypersetup{
    colorlinks,%
    citecolor=black,%
    filecolor=black,%
    linkcolor=black,%
    urlcolor=black
}
\usepackage{subfig}
\numberwithin{equation}{section}

\newcommand{\be}{\begin{equation}}
\newcommand{\ee}{\end{equation}}
\newcommand{\beaa}{\begin{eqnarray*}}
\newcommand{\eeaa}{\end{eqnarray*}}
\newcommand{\bea}{\begin{eqnarray}}
\newcommand{\eea}{\end{eqnarray}}
\newcommand{\lbl}{\label}

\newcommand{\bei}{\begin{itemize}}
\newcommand{\eei}{\end{itemize}}

\newtheorem{theorem}{ \noindent T{\footnotesize HEOREM}}

\newtheorem{lemma}{ \noindent L{\footnotesize EMMA}}[section]

\newtheorem{remark}{ \noindent R{\footnotesize EMARK}}[section]

\begin{document}

\title{A generalized Hardy-Ramanujan formula for the number of restricted integer partitions}
\author{Tiefeng Jiang$^{1}$
and  Ke Wang$^2$\\
University of Minnesota and Hong Kong University of Science and Technology}

\date{}
\maketitle

\footnotetext[1]{School of Statistics, University of Minnesota, 224 Church
Street, S. E., MN55455, USA, jiang040@umn.edu. 
The research of Tiefeng Jiang is
supported in part by NSF Grant DMS-1209166 and DMS-1406279.}

\footnotetext[2]{Department of Mathematics, Hong Kong University of Science and Technology, Clear Water Bay, Kowloon, Hong Kong, kewang@ust.hk. Ke Wang  is supported by by HKUST Initiation Grant IGN16SC05. 
}

\begin{abstract}
\noindent  We derive an asymptotic formula for $p_n(N,M)$, the number of partitions of integer $n$ with part size at most $N$ and length at most $M$. We consider both $N$ and $M$ are comparable to $\sqrt{n}$. This is an extension of the classical Hardy-Ramanujan formula and Szekeres' formula. The proof relies on the saddle point method.
\end{abstract}





\section{Introduction}

A partition of integer $n$ is a sequence of positive integers $\lambda_1\ge \cdots \ge \lambda_k >0$ satisfying $$\sum_{i=1}^k \lambda_i=n,$$ where $k$ is the length and the $\lambda_i$'s are the parts of the partition. Let $p_n$ be the number of all partitions of $n$. In a celebrated paper \cite{HR1918}, Hardy and Ramanujan proved the asymptotic formula 
\begin{align}\label{eq:HR}
p_n \sim \frac{1}{4\sqrt{3} n} \exp\left( \sqrt{\frac{2n}{3}} \pi \right).
\end{align}
For two positive sequences $\{a_n\}$ and $\{b_n\}$, we use $a_n \sim b_n$ if $\lim_{n\to \infty} a_n/b_n=1$. 

Let $p_n(N)$ be the number of partitions of $n$ with part size at most $N$. Note that $p_n(n)=p_n$. Szekeres \cite{Sze51, Sze53} obtained the asymptotic formulas for $p_n(N)$ as $n$ tends to infinity, using substantially different
approaches for two distinct though slightly overlapping ranges of $N$. In particular, Szekeres' formula holds if $\alpha:=N/\sqrt{n} \ge C>0$. The formula was reproduced later by Canfield \cite{Can97} using recursive equations for $p_n(N)$ and Taylor expansion. It is observed in \cite{Can97} that Szekeres' formula could be combined into a single form 
\begin{align}\label{eq:SZ}
p_n(\alpha \sqrt{n}) = \frac{\rho(\alpha) \exp\left[ \Big(2\rho(\alpha)-\alpha \log(1-e^{-\alpha \rho(\alpha)})\Big)\sqrt{n} \right] }{2^{3/2}\pi n \sqrt{1-\left(\frac{\alpha^2}{2}+1\right)e^{-\alpha \rho(\alpha)}}}\left( 1+ O\big(n^{-1/6+\epsilon}\big)\right)
\end{align} 
as $n\to \infty$, where $\rho(\alpha)>0$ is the unique solution to the implicit equation 
\begin{equation}\label{eq:rho}
\rho(\alpha)^2 = \text{Li}_2(1-e^{-\alpha \rho(\alpha)})
\end{equation}
and $\text{Li}_2$ is  \emph{Spence's function}, or \emph{dilogarithm}, defined for 
complex number $|z|<1$ as
\beaa
\text{Li}_2(z)= \sum_{k=1}^{\infty} \frac{z^k}{k^2}.
\eeaa
Since $\rho(\alpha)$ is an increasing function and satisfies
 \[
 \rho(0)=0 \quad\mbox{and}\quad \lim_{\alpha\to\infty}\rho(\alpha)= \frac{\pi}{\sqrt{6}},
 \] 
 it can be checked that when $\alpha\to \infty$, the right side of \eqref{eq:SZ} converges to the right side of \eqref{eq:HR}. Several years later, Romik \cite{Rom05} provided another proof of \eqref{eq:SZ} using probabilistic methods.

In this paper, we focus on $p_n(N,M)$, the number of partitions of $n$ with part size at most $N$ and length at most $M$. The $p_n(N,M)$ for $n=0,\ldots, NM$ are also called the coefficients of the \emph{$q$-binomial coefficients} or \emph{Gaussian binomial coefficients}. In \cite[Theorem 2.4]{SZ16}, the asymptotic behavior of $p_n(N,M)$ has been investigated assuming $M$ is fixed and $N$ gets arbitrarily large. Very recently, Richmond \cite{Ric18} derived the asymptotic formula for $p_n(N,M)$ when both $N$ and $M$ are close to their expected values $\frac{\sqrt{6n}}{\pi} \log (\frac{\sqrt{6n}}{\pi})$ (see \cite{EL41} for the distributions of the length and largest part of a uniform integer partition). When $n$ is close to $NM/2$, an asymptotic formula for $p_n(N,M)$ was obtained before by Tak\'acs \cite{Tak86} . More precisely, whenever $N\to \infty$, $M\to \infty$, $|N-M|=O(\sqrt{M+N})$ and $|n-NM/2|=O(\sqrt{NM(N+M)})$, it was shown in \cite{Tak86} (see also \cite{AA91}) that
\begin{align}\label{eq:AA}
p_n(N,M) = \binom{N+M}{N} \sqrt{\frac{6}{\pi NM (N+M+1)}} \exp\left(-\frac{6(n-{NM}/{2})^2}{NM (N+M+1)}  \right)  \left(1+o(1) \right).
\end{align}

In this paper, we aim to complement this result by deriving asymptotic formulas for $p_n(N,M)$ when $n$ is around $NM/\tau$ with $\tau \neq 2$. We prove such formulas by imposing some extra requirement on $\tau$ (see Theorem \ref{thm:asymp} and discussion in the end of Section \ref{sec:main}). Our motivation to derive such formulas lies in studying the limiting distribution of a partition chosen uniformly from the set of restricted partitions $\mathcal{P}_n(N)$, that is, partitions of $n$ with part size at most $N$, for the entire range $1\le N \le n$. Currently, the uniform distributions on $\mathcal{P}_n(n)$ (see \cite{EL41, Fri93,Pit97}), $\mathcal{P}_n(N)$ for $N$ fixed integer (see \cite{JW16}) and $\mathcal{P}_n(N)$ for $N=o(n^{1/3})$ (see \cite{JW17}) have been studied. We believe the asymptotic formula of $p_n(N,M)$ plays an important role in understanding the uniform distribution on $\mathcal{P}_n(N)$ for other values of $N$. This will be explored in future research.

\subsection{Main result}\label{sec:main}
For positive integers $n, M, N$, we denote $$\alpha=\frac{N}{\sqrt{n}}\quad\mbox{and}\quad \beta=\frac{M}{\sqrt{n}}.$$
\begin{theorem}\lbl{thm:asymp}
 Let $\varepsilon\in(0, 1/8)$ be given.  Then uniformly for $N\ge 4\sqrt{n}$ and $M\ge 4\sqrt{n}$,
we have, as $n\to\infty$, that
\begin{align*}
p_n(N,M) = \frac{g(\alpha,\beta)^2 }{ 2\pi\sqrt{L(\alpha,\beta)}}  \sqrt{\frac{1-e^{-(\alpha+\beta)g(\alpha,\beta)}}{(1-e^{-\alpha g(\alpha,\beta)})(1-e^{-\beta g(\alpha,\beta)})}}\cdot\frac{e^{\sqrt{n}K(\alpha,\beta)}}{n}\Big(1+ O(n^{-1/4+\epsilon})\Big),
\end{align*}
where  $g(\alpha,\beta)>0$ is the unique solution to the implicit equation 
\begin{align*}
g^2(\alpha,\beta) = \text{Li}_2(1-e^{-\alpha g(\alpha,\beta)}) + \text{Li}_2(1-e^{-\beta g(\alpha,\beta)}) - \text{Li}_2 (1-e^{-(\alpha+\beta)g(\alpha,\beta)}),
\end{align*}
and
\begin{align*}
L(\alpha,\beta) &= 2 g^2(\alpha,\beta) + \frac{(\alpha+\beta)^2 g^2(\alpha,\beta)}{e^{(\alpha+\beta)g(\alpha,\beta)}-1} -\frac{\alpha^2 g^2(\alpha,\beta)}{e^{\alpha g(\alpha,\beta)}-1} -\frac{\beta^2 g^2(\alpha,\beta)}{e^{\alpha g(\alpha,\beta)}-1},\\
K(\alpha,\beta) &= g(\alpha,\beta) + \frac{1}{g(\alpha,\beta)}\Big(\frac{\pi^2}{6} + \text{Li}_2(e^{-(\alpha+\beta)g(\alpha,\beta)}) - \text{Li}_2(e^{-\alpha g(\alpha,\beta)}) -\text{Li}_2 (e^{-\beta g(\alpha,\beta)})\Big).
\end{align*}
\end{theorem}
A few remarks regarding Theorem \ref{thm:asymp} are in order.
\begin{remark}
The uniqueness and existence of $g(\alpha,\beta)>0$ is guaranteed by Lemma \ref{lem:changeofvariables}.  
.
\end{remark}

\begin{remark}\label{rem:uniform}
All of $g(\alpha,\beta)$, $L(\alpha,\beta)$, $K(\alpha,\beta)$ and  $ \sqrt{\frac{1-e^{-(\alpha+\beta)g(\alpha,\beta)}}{(1-e^{-\alpha g(\alpha,\beta)})(1-e^{-\beta g(\alpha,\beta)})}}$ in $p_n(N,M)$ in Theorem \ref{thm:asymp} are bounded from above and from below by two universal positive constants  for all  $\alpha,\beta\in[4,\infty]$. The error term $O(n^{-1/4+\epsilon})$ is also uniform in $\alpha,\beta\in[4,\infty]$. 
\end{remark}

\begin{remark}\label{rem:asymptotic}
If we do not put restrictions on either the part size or the length of the partition, then either $\alpha$ or $\beta$ is equal to infinity in our setting, respectively. In this case, the asymptotic formula discovered in Theorem \ref{thm:asymp} is identical to the Szekeres' formula \eqref{eq:SZ}. 

Let us assume $\beta=\infty$ (by \eqref{eq:P1} the same will apply to the case $\alpha=\infty$).  We first notice that $g(\alpha,\infty)=\rho(\alpha)$, where $\rho(\alpha)$ is the one in \eqref{eq:rho}, and secondly,
\begin{align*}
K(\alpha,\infty) &=\rho(\alpha)+ \frac{1}{\rho(\alpha)} \left(\frac{\pi^2}{6}-\text{Li}_2(e^{-\alpha \rho(\alpha)}) \right)\\
&=\rho(\alpha) + \frac{1}{\rho(\alpha)} \big[ \text{Li}_2(1-e^{-\alpha \rho(\alpha)}) - \alpha \rho(\alpha) \log (1-e^{-\alpha \rho(\alpha)} ) \big]\\
&=2 \rho(\alpha) -\alpha  \log (1-e^{-\alpha \rho(\alpha)} ),
\end{align*}
where we use \eqref{eq:otherexp} in the second identity. Also, it is easy to check that 
$$L(\alpha, \infty)= 2 \rho^2(\alpha) \left[ 1- \frac{\alpha^2}{ 2 (e^{\alpha \rho(\alpha)}-1)} \right].$$ 
Therefore, by plugging these into $p_n(N,M)$ in Theorem \ref{thm:asymp}, we will have \eqref{eq:SZ}.
\end{remark}

We derive the asymptotic formula for $p_n(\alpha \sqrt{n},\beta \sqrt{n})$ in Theorem \ref{thm:asymp} assuming $\alpha$ and $\beta$ are greater than 4. Our proof uses the saddle point method; see \cite[Section 12]{Odl95} or \cite[Chapter 8]{FS09} for a detailed introduction to this method. For the case when $\alpha, \beta$ are both small, we could combine Theorem \ref{thm:asymp}  and  $p_n(N,M)=p_{MN-n}(M,N)$ (see \eqref{eq:P2}) to derive the formula, since both $M/\sqrt{MN-n}$ and $N/\sqrt{MN-n}$ will be large. For instance, this would apply if $\max(\alpha,\beta)<\sqrt{16/15}$.

Finally, we conjecture Theorem \ref{thm:asymp} holds for arbitrary $\alpha,\beta>0$ as long as $\alpha\beta=NM/n>2$. For the remaining case $1<\alpha\beta<2$, we can apply the formula for $p_{MN-n}(M,N)=p_n(N,M)$ since $NM/(NM-n) = \alpha\beta/(\alpha\beta-1) >2$. Therefore, combining \eqref{eq:AA} with Theorem \ref{thm:asymp}, the asymptotic formula for $p_n(N,M)$ is clear for arbitrary $N/\sqrt{n} \in (0,+\infty]$ and $M/\sqrt{n} \in (0,+\infty]$. The stronger assumption $\min(\alpha,\beta)\ge 4$ in Theorem  \ref{thm:asymp} appears to be only a technical condition (see Lemma \ref{lem:technical}) and is used to control the $O(n^{-1/4+\epsilon})$ error term in the formula. It could be pushed to, say, $\min(\alpha,\beta)\ge 2.5$ with advanced help of Maple. We did not try to optimize the lower bound of $\min(\alpha,\beta)$.

The paper is organized as follows. In Section \ref{sec:preliminary}, we will introduce some background, and set up our calculation of obtaining the asymptotic formula in Theorem \ref{thm:asymp}. Some analytic lemmas, which will be used later, are also included. In Section \ref{sec:estimate}, we will carry out the detailed calculations. We will first derive the main term, and estimate the error term afterwards.

We note that shortly after a preprint version of this paper appeared on the arXiv, another preprint \cite{MPP} appeared on the arXiv which derives an asymptotic formula for $p_{n}(N,M)$ using a probabilistic approach. In our notations, they obtained an asymptotic formula for $p_{n}(N,M)$ when $\alpha\beta\ge 2$ and assuming $(\alpha, \beta)$ is in an arbitrarily fixed compact set $K\subset\{(x,y)\in\mathbb{R}^2\ |\ x>0,y>0\}$, with an error bound $o(1)$ depending on the compact set $K$. 


\medskip
\noindent{\bf Notations:} We use standard asymptotic notations $o, O$ as $n$ tends to infinity. We denote $a\lesssim b$ if there is a universal positive constant $C$ such that $a\le Cb$.

\medskip
\noindent{\bf Acknowledgement:} We would like to thank the referee for many helpful suggestions to improve the exposition of the paper.

\section{Preliminaries}\label{sec:preliminary}

In this section, we will first recall the generating function of $p_n(N,M)$ and Spence's functions, and show some of their properties that will be used. We will see a connection between the generating function of $p_n(N,M)$ and Spence's functions in Lemma \ref{lem:asymp}. And then, we express $p_n(N,M)$ as an integral using Cauchy's integral formula, and it is the integral that we are going to estimate. In the end of this section, we prepare some lemmas for the detailed proof of Theorem \ref{thm:asymp} in Section \ref{sec:estimate}. The usage of each lemma is explained before each of their statements.

\subsection{Background materials}

We begin with some basic properties of $p_n(N,M)$; see \cite[Chapter 3]{And76} for a comprehensive introduction. Note that $p_n(N,M) = 0$ if $n>M\times N$ and $p_n(N,M)=1$ if $n=M\times N$. Since $p_n(N,M)-p_n(N,M-1)$ counts the number of partitions of $n$ with length exactly $M$ and part size
at most $N$, that is $1\le \lambda_i \le N$ for $1\le i \le M$ and $\sum_{i=1}^M \lambda_i = n$, by considering $\tilde \lambda_i = \lambda_i-1$, it follows easily that
\bea\lbl{eq:rec}
p_n(N,M)-p_n(N,M-1) = p_{n-M}(N-1,M).
\eea

Denote $G(N,M;q)$ the generating function of $p_n(N,M)$ for $n \ge 0$. Thus
\beaa
G(N,M;q)=\sum_{n\ge 0}p_n(N,M) q^n=\sum_{n= 0}^{MN}p_n(N,M) q^n.
\eeaa

The generating function $G(N,M;q)$ has an explicit expression and is usually referred to as the \emph{Gaussian polynomial}. The following lemma can be found at \cite[Theorem 3.1]{And76}. We include the proof for the readers' convenience. 
\begin{lemma}\lbl{lem:generate}For $M,N\ge 0$,
\bea\lbl{eq:generate}
G(N,M;q)=\frac{\prod_{j=1}^{N+M}(1-q^j)}{\prod_{j=1}^N(1-q^j) \prod_{j=1}^M (1-q^j)} = \frac{\prod_{j=M+1}^{N+M}(1-q^j)}{\prod_{j=1}^N(1-q^j)}.
\eea
\end{lemma}
\begin{proof}
It follows from \eqref{eq:rec} that
\beaa
G(N,M;q)=G(N,M-1;q)+q^M\cdot G(N-1,M;q).
\eeaa
Denote the RHS of \eqref{eq:generate} by $\widetilde G(N,M;q)$. By direct calculation, we see that
\beaa
\widetilde G(N,M;q)&=&\frac{\prod_{j=1}^{N+M-1}(1-q^j)}{\prod_{j=1}^N(1-q^j) \prod_{j=1}^{M-1} (1-q^j)}\cdot\frac{1-q^{N+M}}{1-q^M}\\
&=&\widetilde G(N,M-1;q)\cdot \frac{1-q^{M}+q^M(1-q^N)}{1-q^M}\\
&=&\widetilde G(N,M-1;q) + q^M\cdot \widetilde G(N-1,M;q).
\eeaa
Besides, it is easy to check that $G(N,M;q)$ and $\widetilde G(N,M;q)$ satisfy the same initial conditions
\beaa
&&G(N,0;q)=G(0,M;q)=1;\\
&&\widetilde G(N,0;q)=\widetilde G(0,M;q)=1.
\eeaa
Therefore, $G(N,M;q)=\widetilde G(N,M;q)$.
\end{proof}

The following properties of $p_n(N,M)$ can be verified easily from the Ferrers diagram of partitions. We also refer to \cite[Theorem 3.10]{And76} for a proof. 
\begin{lemma}
\begin{align}
&p_n(N,M)=p_n(M,N);\label{eq:P1}\\
&p_n(N,M)=p_{MN-n}(M,N)\label{eq:P2}.
\end{align}
\end{lemma}
It was first proved by Sylvester \cite{Syl1878} that $p_n(N,M)$ is unimodal and 
\begin{equation}\label{eq:pmax}
\max_{1\le n \le NM}p_n(N,M) = p_{[\frac{1}{2}NM]}(N,M).
\end{equation}

Now let us recall  Spence's function, or dilogarithm, which is defined for 
complex numbers $|z|<1$ as
\beaa
\text{Li}_2(z)= \sum_{k=1}^{\infty} \frac{z^k}{k^2}.
\eeaa
It has an analytic continuation for $z\in\mathbb{C}\setminus [1,\infty)$ given by
\beaa
\text{Li}_2(z)=-\int_{0}^z\frac{\log (1-u)}{u}\,du,\ \mbox{or equivalently}\quad\text{Li}_2(1-z)=\int_{1}^z\frac{\log u}{1-u}\,du,
\eeaa
 where $\log$ is the principal branch of the logarithm function. $\text{Li}_2(z)$ can be continuously extended to $z=1$, and $\text{Li}_2(1)= \sum_{k=1}^{\infty} \frac{1}{k^2}=\pi^2/6$. Note that by a change of variable $s=-\log(t)$, we have
\beaa
\text{Li}_2(1-v)=\int_{1}^v\frac{\log t}{1-t}\,dt =\int_{0}^{\log(1/v)}\frac{s}{e^s-1}\,ds.
\eeaa
Equivalently, we have for $x>0$,
\bea\lbl{eq:otherexp}
\int_{0}^{x}\frac{t}{e^t-1}\,dt = \text{Li}_2(1-e^{-x})=\frac{\pi^2}{6} + x\cdot\log(1-e^{-x}) - \text{Li}_2(e^{-x}),
\eea
where the last identity follows from the property that
\begin{align*}
\text{Li}_2(z)+\text{Li}_2(1-z) &= \lim_{\delta\to 0^+}\left(-\int_{0}^{1-\delta}\frac{\log (1-u)}{u}\,du-\int_{1-\delta}^z\frac{\log (1-u)}{u}\,du+\int_{1-\delta}^z\frac{\log u}{1-u}\,du\right)\\
&=\frac{\pi^2}{6}-\lim_{\delta\to 0^+}\int_{1-\delta}^z d\, \Big(\log (1-u)\log u\Big)\\
&=\frac{\pi^2}{6} -\log(z)\cdot\log(1-z).
\end{align*}

\subsection{Setting up of the calculation}

By Cauchy's integral formula,
$$p_n(N,M)=\frac{1}{2\pi i} \int_{|z|=r} \frac{G(N,M;z)}{z^{n+1}}\,dz$$
for any $0<r<1$. We substitute $z=e^{-(v+iw)}$ and for convenience, denote
\bea\lbl{eq:function}
f(v+iw) = f_{N,M}(v+iw)=G(N,M;e^{-(v+iw)}) = \frac{\prod_{j=M+1}^{N+M}(1-e^{-j(v+iw)})}{\prod_{j=1}^N(1-e^{-j(v+iw)})}.
\eea
Therefore,
\begin{equation}\label{eq:cauchyintegral}
p_n(N,M)=\frac{1}{2\pi}\int_{-\pi}^{\pi} f(v+iw) e^{n(v+iw)}\,dw=\frac{1}{2\pi}\int_{-\pi}^{\pi} e^{\log f(v+iw)} e^{n(v+iw)}\,dw
\end{equation}
for every $v>0$. In the end, for our purpose, we will choose $v=c \cdot n^{1/2}$ for some constant $c>0$.  In Section \ref{sec:estimate}, we will show that the main term in the above integral is 
\begin{equation}\label{eq:firstorderaux}
\frac{1}{2\pi}\int_{-w_0}^{w_0} e^{\log f(v+iw)} e^{n(v+iw)}\,dw
\end{equation}
with $w_0=n^{-3/4+\epsilon/3}$, and we will prove that what is left in the integral is a lower order term. 

\subsection{Supporting lemmas}
To estimate \eqref{eq:firstorderaux}, we shall first analyze the function $\log f(z)$. 

\begin{lemma}\lbl{lem:asymp}
 Assume $z=v+iw$ with $v>0$, $|z|<1$ and assume $z$ stays within some angle in the right half plane, that is, $v/|z|=\text{Re}(z)/|z| \ge \lambda$ for some $\lambda>0$. Recall $f(z)$ in \eqref{eq:function}. Then for $ |z|\ge \frac{\bar c}{\min(M,N)}$ with some $\bar c>0$, we have 
 \beaa
 \log f(z) &=& \frac{1}{z} \left(\frac{\pi^2}{6} + \text{Li}_2(e^{-z(N+M)}) - \text{Li}_2(e^{-zN})-\text{Li}_2(e^{-zM})\right)\\
 && -\frac{1}{2}\log\left(\frac{1}{z}\right)-\frac{1}{2}\log \frac{(1-e^{-zN})(1-e^{-zM})}{1-e^{-z(N+M)}}-\log\sqrt{2\pi} + O(|z|),
 \eeaa
 where $\log$ is the principal branch of the logarithm function.
\end{lemma}

\begin{proof}
Since
\beaa
\log f(z) = \sum_{j=M+1}^{M+N} \log(1-e^{-jz}) - \sum_{j=1}^N  \log(1-e^{-jz}),
\eeaa
by Taylor expansion,
\beaa
\log f(z) &=& -\sum_{j=M+1}^{M+N}\sum_{k=1}^{\infty} \frac{e^{-jkz}}{k} + \sum_{j=1}^N \sum_{k=1}^{\infty} \frac{e^{-jkz}}{k}\\
&=&\sum_{k=1}^{\infty} \frac{1}{k}\left(\sum_{j=1}^N e^{-kz\cdot j} - \sum_{j=M+1}^{M+N} e^{-kz\cdot j}\right)\\
&=& \sum_{k=1}^{\infty} \frac{1}{k} \frac{e^{-kz}(1-e^{-kz\cdot N})(1-e^{-kz\cdot M})}{1-e^{-kz}}\\
&=& z\cdot \sum_{k=1}^{\infty} \frac{1}{kz} \frac{(1-e^{-kz\cdot N})(1-e^{-kz\cdot M})}{e^{kz}-1}.
\eeaa
Now we compare $\log f(z)$ with another summation. 
\beaa
&& \log f(z) - z\cdot \sum_{k=1}^{\infty} (1-e^{-kz\cdot N})(1-e^{-kz\cdot M})\left( \frac{1}{k^2z^2} - \frac{e^{-kz}}{2kz} \right)\\
&=&z\cdot \sum_{k=1}^{\infty} (1-e^{-kz\cdot N})(1-e^{-kz\cdot M})\left(\frac{1}{kz(e^{kz}-1)} - \frac{1}{k^2 z^2} + \frac{e^{-kz}}{2kz} \right).
\eeaa
Define $\psi(z)= \psi_1(z)\cdot \psi_2(z)$ where 
\beaa
&&\psi_1(z)=(1-e^{-Nz})(1-e^{-Mz});\\
&&\psi_2(z)=\frac{1}{z(e^{z}-1)} - \frac{1}{z^2} + \frac{e^{-z}}{2z}.
\eeaa
It is easy to verify that the Laurent expansion of $\psi_2(z)$ is equal to $-\frac{5}{12}+\lambda(z)$ where $\lambda(z)$ is analytic and $\lambda(0)=0$. Thus $\psi_2(z)$ and $\psi(z)$ are analytic in the considered region.

Now we can express 
\bea\lbl{eq:logsum}
\log f(z) - z\cdot \sum_{k=1}^{\infty} (1-e^{-kz\cdot N})(1-e^{-kz\cdot M})\left( \frac{1}{k^2 z^2} - \frac{e^{-kz}}{2kz} \right) = \sum_{k=1}^{\infty} z\psi(kz).
\eea
Let us compute the summation on the RHS of \eqref{eq:logsum}. It is natural to use an integral to approximate the summation and we control the difference. 
Denote $\gamma$ the ray starting from origin and passing through $z$. It is clear the integral $\int_{\gamma} \psi(u)\,du$ is convergent. We are going to show that
\bea\lbl{eq:eqint}
\left|\sum_{k=1}^{\infty}z \psi(kz)-\int_{\gamma} \psi(u)\,du\right| = O(|z|).
\eea
Let $\gamma_k$ be the line segment on $\gamma$ from the point $(k-1)z$ to the point $kz$ for $k\ge 1$. Therefore, we have
\beaa
\left|\sum_{k=1}^{\infty} z\psi(k z) - \int_\gamma \psi(u)\,du\right| &\le & \left|z \psi(z) - \int_{\gamma_1} \psi(u)\, du\right|+ \left|\sum_{k=2}^{\infty}\left(z\psi(kz) - \int_{\gamma_k} \psi(u)  \,du\right) \right|.
\eeaa
Since $\psi(z)$ is analytic at $z=0$, we estimate the first term by
\beaa
\left|z \psi(z) - \int_{\gamma_1} \psi(u)\, du\right| = O(|z|).
\eeaa
For the other term, by integration by parts, we have
\beaa
\int_{\gamma_k} \psi(u)\,du &=& kz \psi(kz)-(k-1)z\psi((k-1)z) -  \int_{\gamma_k} u \psi'(u)\,du\\
&=& z\psi(kz) + \int_{\gamma_k} \big((k-1)z -u\big)\psi'(u)\,du.
\eeaa
Thus 
\beaa
\left|z\psi(kz) - \int_{\gamma_k} \psi(u)  \,du\right|  &=& \left|\int_{\gamma_k} \big((k-1)z -u\big)\psi'(u)\,du \right|\\
&\le& |z| \max_{u \in \gamma_k}\, (|u-(k-1)z| \cdot |\psi'(u)|)\\
&\le & |z|^2 \max_{u \in \gamma_k}|\psi'(u)| \le |z|^2 \max_{u \in \gamma_k} (|\psi_1'(u)\psi_2(u)| +|\psi_1(u)\psi_2'(u)|).
\eeaa
Since $|\psi_1(u)|$ and $|\psi_2(u)|$ are universally  bounded on $\gamma_k$ for $k\ge 2$, to show \eqref{eq:eqint}, it is enough to show that for $i=1,2$,
\begin{equation}\label{eq:estimateofphi}
|z|^2 \cdot \sum_{k=2}^{\infty} \max_{u \in \gamma_k} |\psi_i'(u)|=O(|z|).
\end{equation}
In the below we will first prove \eqref{eq:estimateofphi} for $\psi_2$, and then for $\psi_1$.

Note that 
\beaa
\psi_2'(u) = -\frac{1}{u^2(e^u-1)}-\frac{1}{u(e^u-1)}-\frac{1}{u^2(e^u-1)^2} + \frac{2}{u^3} - \frac{1}{u e^u}-\frac{1}{u^2 e^u}.
\eeaa
Recall $v/|z|=\text{Re}(z)/|z| \ge \lambda$. Since $|\psi_2'(u)|$ is analytic and thus bounded for $u \in \gamma_k$ where $2\le k \le 1/(\lambda|z|) +1$, it follows that 
\beaa
|z|^2 \cdot  \sum_{2\le k \le 1/\lambda|z| +1} \max_{u \in \gamma_k} |\psi_2'(u)| = O\left(|z|^2 \cdot \frac{1}{|z|}\right)=O(|z|).
\eeaa
For $k > 1/\lambda|z| +1$, we have $e^{(k-1)v} > e^{(k-1)\lambda|z|} > e$ and thus $e^{(k-1)v}-1\ge e^{(k-1)v}/2$. Therefore,
\beaa
&&|z|^2 \cdot  \sum_{k > 1/\lambda|z| +1} \max_{u \in \gamma_k} |\psi_2'(u)| \\
&\le& 2|z|^2 \cdot  \sum_{k > 1/\lambda|z| +1} \max_{u \in \gamma_k} \frac{1}{|u|^2 |e^u-1|}+ |z|^2 \cdot  \sum_{k > 1/\lambda|z| +1} \max_{u \in \gamma_k}\frac{1}{|u|^2 |e^u-1|^2}\\
&&\quad + 4|z|^2 \cdot  \sum_{k > 1/\lambda|z| +1} \max_{u \in \gamma_k}\frac{1}{|u|^3}\\
&\le& 6|z|^2 \cdot \sum_{k > 1/\lambda|z| +1}\frac{1}{(k-1)^2|z|^2 e^{(k-1)v}} + 4|z|^2 \cdot \sum_{k > \lambda|z| +1} \frac{1}{(k-1)^3|z|^3}\\
&\le&  \sum_{k > 1/\lambda|z|}\frac{6}{ e^{kv}} + \frac{4}{|z|} \sum_{k > 1/\lambda|z|} \frac{1}{k^3} =O(|z|),
\eeaa
where the last inequality follows from the estimates
\beaa
&&\sum_{k > 1/\lambda|z|}\frac{1}{e^{k v}} \le \sum_{k=0}^{\infty} \frac{1}{e^{kv}} = \frac{1}{1-e^{-v}}=O\left(\frac{1}{v}\right)=O\left(\frac{1}{|z|}\right);\\
&& \sum_{k > 1/\lambda |z|} \frac{1}{k^3} \le \int_{1/\lambda |z|}^{\infty} \frac{1}{x^3}\,dx =O(|z|^2).
\eeaa
Hence, we proved \eqref{eq:estimateofphi} for $\psi_2$. 

Now, let us prove \eqref{eq:estimateofphi} for $\psi_1$.  Since $\min_{u \in \gamma_k} |e^{Nu}| =\min_{u\in \gamma_k} e^{N \cdot \text{Re}(u)} \ge e^{(k-1)N\cdot \text{Re}(z)}=e^{(k-1)Nv}$, we obtain
\beaa
|z|^2 \cdot \sum_{k=2}^{\infty} \max_{u \in \gamma_k} |\psi_1'(u)|&=& |z|^2 \cdot \sum_{k=2}^{\infty} \max_{u \in \gamma_k}\left(N\frac{|1-e^{-Mu}|}{|e^{Nu}|}+M\frac{|1-e^{-Nu}|}{|e^{Mu}|}\right)\\
&\le& 2|z|^2 \cdot \sum_{k=2}^{\infty} \left(\frac{N}{e^{Nv(k-1)}}+ \frac{M}{e^{Mv(k-1)}}\right) \\
&\le & \frac{2|z|}{\lambda} \cdot \left(\frac{ N v \cdot e^{-vN}}{1-e^{-vN}}+\frac{Mv \cdot e^{-vM}}{1-e^{-v M}} \right) = O(|z|)
\eeaa
by the assumption $|z| \le v/\lambda$, where we used the fact that the function $\frac{x}{e^x-1}$ is bounded on $(0,\infty)$. 

This finishes the proof of \eqref{eq:estimateofphi}. So, \eqref{eq:eqint} is proved.

Since $\psi(u)$ is analytic in the considered region, by Cauchy integral theorem, it is standard to show that 
$$\int_{\gamma} \psi(u)\, du = \int_0^{\infty} \psi(x)\, dx.$$
Combining \eqref{eq:eqint} and \eqref{eq:logsum}, we obtain
\begin{equation}\label{eq:logvalue}
\log f(z)= z\cdot \sum_{k=1}^{\infty} (1-e^{-kz\cdot N})(1-e^{-kz\cdot M})\left( \frac{1}{k^2z^2} - \frac{e^{-kz}}{2kz} \right) + \int_{0}^{\infty} \psi(x)\,dx + O(|z|).
\end{equation}
It remains to estimate the first two terms on the RHS of \eqref{eq:logvalue}. 

We split the first term into two parts, 
\beaa
&&z\cdot \sum_{k=1}^{\infty} (1-e^{-kz\cdot N})(1-e^{-kz\cdot M})\left( \frac{1}{k^2z^2}- \frac{e^{-kz}}{2kz} \right) \\
&=& \frac{1}{z}\sum_{k=1}^{\infty} \frac{(1-e^{-kz\cdot N})(1-e^{-kz\cdot M})}{k^2} -\frac{1}{2}\sum_{k=1}^{\infty}\frac{e^{-kz}(1-e^{-kz\cdot N})(1-e^{-kz\cdot M})}{k}.
\eeaa
By the definition of Spence's function, we have 
\begin{align*}
 \frac{1}{z}\sum_{k=1}^{\infty} \frac{(1-e^{-kz\cdot N})(1-e^{-kz\cdot M})}{k^2}&=\frac{1}{z} \sum_{k=1}^{\infty} \left(\frac{1}{k^2} -\frac{e^{-k\cdot zN}}{k^2}-\frac{e^{-k\cdot zM}}{k^2}+\frac{e^{-k\cdot z(N+M)}}{k^2}\right)\\
&=\frac{1}{z} \left(\frac{\pi^2}{6} + \text{Li}_2(e^{-z(N+M)}) - \text{Li}_2(e^{-zN})-\text{Li}_2(e^{-zM})\right).
\end{align*}
Since $\sum_{k=1}^{\infty} e^{-k\cdot z}/k = -\log(1-e^{-z})$, we have
\beaa
&&\frac{1}{2}\sum_{k=1}^{\infty}\frac{e^{-kz}(1-e^{-kz\cdot N})(1-e^{-kz\cdot M})}{k}\\
&=&-\frac{1}{2}\log(1-e^{-z})-\frac{1}{2}\sum_{k=1}^{\infty}\frac{1}{k}\left(e^{-kz(N+1)}+e^{-kz(M+1)}-e^{-kz(M+N+1)}\right)\\
&=&-\frac{1}{2}\log(1-e^{-z}) + \frac{1}{2}\log\frac{(1-e^{-z(N+1)})(1-e^{-z(M+1)})}{1-e^{-z(M+N+1)}}\\
&=&\frac{1}{2}\log\left(\frac{1}{z}\right)  + \frac{1}{2}\log\frac{(1-e^{-zN})(1-e^{-zM})}{1-e^{-z(M+N)}}+O(|z|),
\eeaa
where the last identity follows from the fact that $\lim_{u\to0} \frac{1}{u} \log(\frac{1-e^{-u}}{u})$ exists, and thus, the function $\frac{1}{u} \log(\frac{1-e^{-u}}{u})$ is bounded around $u=0$, implying that
\beaa
\frac{1}{2}\log\left(\frac{1}{z}\right) = -\frac{1}{2}\log(1-e^{-z}) +O(|z|).
\eeaa
Hence, we have for the first term on the RHS of \eqref{eq:logvalue} that
\begin{equation}\label{eq:logvalueinfiniteseries}
\begin{split}
&z\cdot \sum_{k=1}^{\infty} (1-e^{-kz\cdot N})(1-e^{-kz\cdot M})\left( \frac{1}{k^2z^2}- \frac{e^{-kz}}{2kz} \right)\\
&=\frac{1}{z} \left(\frac{\pi^2}{6} + \text{Li}_2(e^{-z(N+M)}) - \text{Li}_2(e^{-zN})-\text{Li}_2(e^{-zM})\right)\\
&\quad-\frac{1}{2}\log\left(\frac{1}{z}\right)  - \frac{1}{2}\log\frac{(1-e^{-zN})(1-e^{-zM})}{1-e^{-z(M+N)}}+O(|z|).
\end{split}
\end{equation}

Now we calculate the second term  on the RHS of \eqref{eq:logvalue}. We have
\beaa
\int_{0}^{\infty} \psi(x) \,dx &=& \int_{0}^{\infty} \psi_2(x) \,dx + \int_0^{\infty} (e^{-(N+M)x}-e^{-Nx}-e^{-Mx})\psi_2(x)\,dx\\
&=&-\log\sqrt{2\pi} + \int_0^{\infty} (e^{-(N+M)x}-e^{-Nx}-e^{-Mx})\psi_2(x)\,dx,
\eeaa
where the calculation of $\int_{0}^{\infty} \psi_2(x) \,dx=-\log\sqrt{2\pi}$ is given in \cite[Page 125]{Pos88}. Since $\psi_2(x)$ is bounded, we have
\beaa
\left|\int_0^{\infty} (e^{-(N+M)x}-e^{-Nx}-e^{-Mx})\psi_2(x)\,dx\right| &=& O\left(\int_{0}^{\infty} (e^{-(N+M)x}+e^{-Nx}+e^{-Mx})\,dx\right)\\
&=&O\left(\frac{1}{M+N} + \frac{1}{N} + \frac{1}{M} \right) = O(|z|).
\eeaa
Hence,
\begin{equation}\label{eq:logvaluelastterm}
\int_{0}^{\infty} \psi(x) \,dx =-\log\sqrt{2\pi}+O(|z|).
\end{equation}

Putting the equations \eqref{eq:logvalue}, \eqref{eq:logvalueinfiniteseries} and \eqref{eq:logvaluelastterm} together, we prove Lemma \ref{lem:asymp}.
\end{proof}

We will use the following lemma to change the variables in our calculations. Instead of using directly the variables $\alpha$ and $\beta$ in Theorem \ref{thm:asymp}, we will use  the variables $c_1$ and $c_2$ defined in \eqref{eq:c1c2} in the next lemma, which will significantly simplify our calculations. So we will derive our asymptotic formula of $p_n(N,M)$ in terms of $c_1$ and $c_2$ first, and  change back to $\alpha$ and $\beta$ in the end.

\begin{lemma}\label{lem:changeofvariables}
For every $\alpha,\beta\in (0,\infty)$ such that $\alpha\beta>2$, there exist unique $c_1,c_2>0$ satisfying the equations:
\begin{equation}\label{eq:c1c2}
c_1=\alpha \cdot A(c_1, c_2), \quad c_2 = \beta \cdot A(c_1,c_2),
\end{equation}
where 
\begin{equation}\label{eq:A}
A(c_1,c_2)=\sqrt{\int_{0}^{c_1} \frac{x}{e^x-1}\,dx -\int_{c_2}^{c_1+c_2} \frac{x}{e^x-1}\,dx}.
\end{equation}
Moreover, if $\min(\alpha,\beta)\ge\sqrt{2}+\delta$ for some $\delta>0$, then there exist two constants $\overline A>0$ and $ \bar c>0$ both depending only on $\delta$ such that
\begin{equation}\label{eq:lowerboundofc}
\min(c_1,c_2)\ge \bar c\quad\mbox{and}\quad A(c_1,c_2)\ge\overline A.
\end{equation}
In particular, if $\min(\alpha,\beta)\ge 4$, then $\min(c_1,c_2)\ge 13/5$.

\end{lemma}
\begin{proof}
Note that the quantity inside of the square root of \eqref{eq:A} is positive since the function $\frac{x}{e^x-1}$ is strictly decreasing for $x\ge 0$. 
Denote $\lambda=\beta/\alpha$ for brevity. Consider $$b(t) =A(t,\lambda t)^2 - (t/\alpha)^2,\quad t>0.$$


We have
\beaa
b'(t)&=&\frac{t}{e^t-1}-\frac{\lambda^2 t}{e^{\lambda t}-1}-\frac{(1+\lambda)^2t}{e^{(1+\lambda)t}-1}-\frac{2t}{\alpha^2}\\
&=& \phi(t)+\lambda \phi(\lambda t)-(1+\lambda)\phi((1+\lambda)t),
\eeaa
where
\[
\phi(t)=\frac{t}{e^t-1}+\frac{t}{\alpha\beta}.
\]
It is elementary to check that $\phi$ is convex for $t\ge 0$, and thus, $\phi'$ is strictly increasing. Moreover, one can also verify that $\phi'(0)=-\frac 12+\frac{1}{\alpha\beta}<0$ since $\alpha\beta>2$, and $\phi'(t)>0$ if $t$ is large. Hence, there exists $t_0>0$ such that $\phi'(t)<0$ in $(0,t_0)$ and $\phi'(t)>0$ in $(t_0,\infty)$.  That is, $\phi(t)$ is strictly decreasing in $(0,t_0)$. Therefore, if $0<(1+\lambda)t\le t_0$, then
\[
b'(t)=\phi(t)+\lambda \phi(\lambda t)-(1+\lambda)\phi((1+\lambda)t)>0.
\]
If $(1+\lambda)t> t_0$, since $\phi'$ is strictly increasing, we have
\beaa
b''(t)&=&\phi'(t)+\lambda^2 \phi'(\lambda t)-(1+\lambda)^2\phi'((1+\lambda)t)\\
&<&\phi'((1+\lambda)t)+\lambda^2\phi'((1+\lambda)t)-(1+\lambda)^2\phi'((1+\lambda)t)<0.
\eeaa
Hence, we have that $b'(t)$ is positive when $0< t\le \frac{t_0}{1+\lambda}$, and then is decreasing for $t>\frac{t_0}{1+\lambda}$. Since $b'(t)\to-\infty$ as $t\to\infty$, then we know that there exists $\bar t$ such that $b'(t)>0$ if $t\in (0,\bar t)$ and $b'(t)<0$ if $t\in (\bar t,\infty)$. That is, $b(t)$ is strictly increasing when $t\in (0,\bar t)$ and $b(t)$ is strictly decreasing when $t\in (\bar t,\infty)$. Since $b(0)=0$ and $\lim_{t\to +\infty} b(t) = -\infty$, there exists a unique $c_1>0$ such that $b(c_1)=0$. This proves the existence and uniqueness of $c_1,c_2>0$, which are solutions to \eqref{eq:c1c2}. 

To prove \eqref{eq:lowerboundofc}, without loss of generality, we assume $\alpha\le\beta$ and thus $c_1\le c_2$.  Since the function $\frac{x}{e^x-1}$ is strictly decreasing, we observe that
\[
c_1^2=\alpha^2 A(c_1,c_2)\ge\alpha^2\left(\int_{0}^{c_1} \frac{x}{e^x-1}\,dx -\int_{c_1}^{2c_1} \frac{x}{e^x-1}\,dx\right).
\]
That is
\[
\alpha^2\le\frac{c_1^2}{\int_{0}^{c_1} \frac{x}{e^x-1}\,dx -\int_{c_1}^{2c_1} \frac{x}{e^x-1}\,dx}\to 2\quad\mbox{as }c_1\to 0^+.
\]
Therefore, if $\alpha\ge\sqrt{2}+\delta$ for some $\delta>0$, then $c_1\ge \bar c$ for some $\bar c>0$ depending only on $\delta$. Moreover, for 
\[
B(c_1):= \int_{0}^{c_1} \frac{x}{e^x-1}\,dx -\int_{c_1}^{2c_1} \frac{x}{e^x-1}\,dx,
\]
we have
\[
\frac{d}{d c_1}B(c_1)=\frac{2c_1}{e^{c_1}+1}>0.
\]
Therefore, $B(c_1)$ is increasing, and thus, for all $c_2\ge c_1\ge\bar c$, we have
\[
A(c_1,c_2)\ge \sqrt{B(c_1)}\ge  \sqrt{\int_{0}^{\bar c} \frac{x}{e^x-1}\,dx -\int_{\bar c}^{2\bar c} \frac{x}{e^x-1}\,dx}=:\overline A>0.
\]

Now we will show that $\min(c_1,c_2)\ge 13/5$ if we suppose $\min(\alpha,\beta)\ge 4$. Without loss of generality, we assume $\alpha\le\beta$. Hence, $\lambda=\beta/\alpha\ge 1$. Since $\frac{x}{e^x-1}$ is decreasing, 
\begin{align*}
b(t)\ge A(t, t)^2 - (t/\alpha)^2 &= \int_0^t\frac{x}{e^x-1}\,dx-\int_{t}^{2t}\frac{x}{e^x-1}\,dx-\frac{t^2}{\alpha^2} \\
&=\frac{t^2}{e^t-1}-\frac{2t^2}{e^{2t}-1} + \int_0^t\frac{x^2e^x}{(e^x-1)^2}\,dx-\int_t^{2t}\frac{x^2e^x}{(e^x-1)^2}\,dx-\frac{t^2}{\alpha^2}\\
&\ge \frac{t^2}{e^t-1}-\frac{2t^2}{e^{2t}-1}-\frac{t^2}{\alpha^2},
\end{align*}
where we used integration by parts in the second equality, and the fact that $\frac{x^2e^x}{(e^x-1)^2}$ is decreasing in $x$ in the last inequality. From this, one can verify that if $\alpha \ge 4$, then $b(13/5)>0$. It is clear that $b(0)=0$ and $\lim_{t\to-\infty}b(t)=-\infty$. Moreover, we know from the proof of Lemma \ref{lem:changeofvariables} that $b(t)$ (which is the same as the $b(t)$ there) is strictly increasing in $(0,\bar t)$ and strictly decreasing in $(\bar t,\infty)$ for some $\bar t>0$. Therefore, we know that the unique solution $c_1$ of $b(c_1)=0$ satisfies that $(c_2>)c_1>13/5$.
\end{proof}

By \eqref{eq:otherexp}, we also can express
\begin{align}\lbl{eq:A2}
A(c_1,c_2)^2 &= \text{Li}_2(1-e^{-c_1})+\text{Li}_2(1-e^{-c_2})-\text{Li}_2(1-e^{-c_1-c_2})\nonumber\\
&=\frac{\pi^2}{6} +\text{Li}_2(e^{-c_1-c_2})-\text{Li}_2(e^{-c_1})-\text{Li}_2(e^{-c_2})\nonumber\\
&\quad-(c_1+c_2)\log(1-e^{-c_1-c_2})+c_1\log(1-e^{-c_1})+c_2\log(1-e^{-c_2}).
\end{align}
Therefore,
\begin{equation}\label{eq:A3}
\begin{split}
&\frac{\pi^2}{6} +\text{Li}_2(e^{-c_1-c_2})-\text{Li}_2(e^{-c_1})-\text{Li}_2(e^{-c_2})\\
&= A(c_1,c_2)^2 +c_1\log\left(\frac{1-e^{-c_1-c_2}}{1-e^{-c_1}}\right)+c_2\log\left(\frac{1-e^{-c_1-c_2}}{1-e^{-c_2}}\right)\ge A(c_1,c_2)^2. 
\end{split}
\end{equation}
These two formulae \eqref{eq:A2} and \eqref{eq:A3} will be used later.

The next technical lemma will be used to estimate the error term in Section \ref{subsec:error}, which we include here for convenience. For a complex number $z$, we denote $\text{Re}(z)$ as the real part of $z$, and $\text{Im}(z)$ as the imaginary part of $z$.
\begin{lemma}\label{lem:technical}
Suppose $\min(\alpha,\beta)\ge 4$, and let $c_1,c_2$ be as in Lemma \ref{lem:changeofvariables}. Define two functions:
\begin{align}
\eta_{c}(x) &= \text{Re}(\text{Li}_2(e^{-c-ic x})) + x\cdot \text{Im}(\text{Li}_2(e^{-c-ic x})),\label{eq:definitionf}\\
H(x)&=\frac{1}{1+x^2} \left( \frac{\pi^2}{6} +\eta_{c_1+c_2}(x) - \eta_{c_1}(x) -\eta_{c_2}(x)\right).\label{eq:definitionG}
\end{align}
Then there exists a universal constant $\delta>0$ such that
\[
H(x)\le H(0)-\frac{\delta x^2}{1+x^2}\quad\mbox{for all }x>0.
\]
\end{lemma}
\begin{proof}
For convenience, denote $$F(x):= \frac{\pi^2}{6} +\eta_{c_1+c_2}(x) - \eta_{c_1}(x) -\eta_{c_2}(x).$$ 
Then $F(0)=H(0)$. We are going to show there exists a universal constant $\delta>0$ such that
\bea\lbl{eq:ineq}
F'(x) \le 2F(0)x -2\delta x\quad\mbox{for all }x>0,
\eea
from which it follows by integrating both sides that 
\[
H(x)\le H(0)-\frac{\delta x^2}{1+x^2}\quad\mbox{for all }x>0.
\]

We have
\begin{align*}
F(x)&=\frac{\pi^2}{6}+\sum_{k=1}^\infty\frac{e^{-(c_1+c_2)k}}{k^2}\Big(\cos(c_1+c_2)kx-x\sin(c_1+c_2)kx\Big)\\
&-\sum_{k=1}^\infty\frac{e^{-c_1k}}{k^2}\Big(\cos (c_1kx)-x\sin(c_1kx)\Big)-\sum_{k=1}^\infty\frac{e^{-c_2k}}{k^2}\Big(\cos (c_2kx)-x\sin(c_2kx)\Big).
\end{align*}
Therefore,
\begin{align*}
&F'(x)\\
&=\sum_{k=1}^\infty\frac{e^{-(c_1+c_2)k}}{k^2}\Big(-(c_1+c_2)k\sin(c_1+c_2)kx-\sin(c_1+c_2)kx-(c_1+c_2)kx\cos(c_1+c_2)kx\Big)\\
&\quad+\sum_{k=1}^\infty\frac{e^{-c_1k}}{k^2}\Big(c_1k\sin (c_1kx)-\sin(c_1kx)+c_1kx \cos(c_1kx)\Big)\\
&\quad+\sum_{k=1}^\infty\frac{e^{-c_2k}}{k^2}\Big(c_2k\sin (c_2kx)-\sin(c_2kx)+c_2kx \cos(c_2kx)\Big)\\
&\le \sum_{k=1}^\infty\frac{e^{-(c_1+c_2)k}}{k^2}\Big((c_1+c_2)^2k^2x+2(c_1+c_2)kx\Big)+\sum_{k=1}^\infty\frac{e^{-c_1k}}{k^2}\Big(c_1^2k^2x+2c_1kx\Big)\\
&\quad+\sum_{k=1}^\infty\frac{e^{-c_2k}}{k^2}\Big(c_2^2k^2x+2c_2kx\Big)\\
&=x\Big(\frac{(c_1+c_2)^2}{e^{c_1+c_2}-1}-2(c_1+c_2)\log(1-e^{-(c_1+c_2)})\\
&\quad+\frac{c_1^2}{e^{c_1}-1}-2c_1\log(1-e^{-c_1})+\frac{c_2^2}{e^{c_2}-1}-2c_2\log(1-e^{-c_2})\Big).
\end{align*}


Let 
\[
\zeta(c)=\frac{c^2}{e^{c}-1}-2c\log(1-e^{-c})+2\text{Li}_2(e^{- c}).
\]
Then $\zeta(0)=\pi^2/3$ and $\lim_{c\to\infty}\zeta(c)=0$. Also, $\zeta$ is strictly decreasing since
\begin{align*}
\zeta'(c)=-\frac{c^2e^{c}}{(e^{c}-1)^2}<0.
\end{align*}
Therefore, there exists $\bar c$ to be the unique solution of $\zeta(\bar c)=\frac{111\pi^2}{1000}$. Then we have
\begin{align}
\frac{F'(x)}{x}-2F(0)&\le \zeta(c_1+c_2)+\zeta(c_1)+\zeta(c_2)-\frac{\pi^2}{3}-4\text{Li}_2(e^{-(c_1+c_2)}) \nonumber\\
&\le  3 \zeta(\bar c)-\frac{\pi^2}{3}\nonumber\\
&=-\frac{\pi^2}{3000}=:-2\delta\label{eq:choiceofdelta}
\end{align}
as long as $\min(c_1,c_2)\ge \bar c$. 

Let us estimate the value of $\bar c$. Notice that
\[
\text{Li}_2(e^{- c})=\sum_{k=1}^\infty\frac{e^{-ck}}{k^2}<e^{-c}+\sum_{k=2}^\infty\frac{e^{-ck}}{2^2}= e^{-c}+\frac{e^{-2c}}{4(1-e^{-c})}=\frac{4e^c-3}{4(e^{2c}-e^c)}.
\]
Hence
\[
\zeta(c)<\frac{c^2}{e^{c}-1}-2c\log(1-e^{-c})+\frac{4e^c-3}{2(e^{2c}-e^c)}.
\]
Using the above estimate, we can verify that $\zeta(13/5)<111\pi^2/1000$. Since $\zeta$ is  decreasing, we have $\bar c<13/5$. 

Since we assume that $\min(\alpha,\beta)\ge 4$, we know from Lemma \ref{lem:changeofvariables} that $\min(c_1,c_2)\ge 13/5$. Therefore,  $\min(c_1,c_2)\ge \bar c$, and thus, \eqref{eq:choiceofdelta} holds.  This proves Lemma \ref{lem:technical}.
\end{proof}

\section{Proof of the main result}\label{sec:estimate}
Now we will start to prove our main result.

\begin{proof}[Proof of Theorem \ref{thm:asymp}]
Since $p_n(N,M)=p_n(M,N)$, without loss of generality we can assume $N \le M$. Recall
\[
\alpha=\frac{N}{\sqrt{n}},\quad\beta=\frac {M}{\sqrt{n}},
\]
and also recall our assumption that $\min(\alpha,\beta)\ge 4$. Let $c_1,c_2$ be the solutions of \eqref{eq:c1c2} with $A(c_1,c_2)$ defined in \eqref{eq:A}. Set 
\beaa
v=\frac{A(c_1,c_2)}{\sqrt{n}} \quad \mbox{and} \quad w_0 =\frac{1}{n^{\frac{3}{4} - \frac{\epsilon}{3}}}.
\eeaa
Then $N \cdot v = c_1$ and $M\cdot v = c_2$. First, we observe from the definition of $A(c_1,c_2)$ in \eqref{eq:A} that
\begin{equation}\label{eq:Aboundabove}
A(c_1,c_2)\le \sqrt{\int_{0}^{\infty} \frac{x}{e^x-1}\,dx}<\infty.
\end{equation}
Secondly, since we have assumed that $\alpha, \beta\ge 4$, we know from Lemma \ref{lem:changeofvariables} that there exist two positive universal constants $\bar c$ and $\overline A$ such that
\begin{equation}\label{eq:lowerboundcA}
\min(c_1,c_2)\ge \bar c>0\quad\mbox{and}\quad A(c_1,c_2)\ge\overline A>0.
\end{equation}
These two universal lower bounds are important to prove the uniform estimate in the next subsection. Since we assumed that $N\le M$ at the beginning of this proof, it follows that $c_1\le c_2$.

Recall the integral \eqref{eq:cauchyintegral} that
\bea
p_n(N,M)&=&\frac{1}{2\pi}\int_{-\pi}^{\pi} e^{\log f(v+iw)}  e^{n(v+iw)}\,dw \nonumber\\
&=& \frac{1}{2\pi}\int_{-w_0}^{w_0} e^{\log f(v+iw)}  e^{n(v+iw)}\,dw + \frac{1}{2\pi}\int_{w_0}^{\pi} e^{\log f(v+iw)}  e^{n(v+iw)}\,dw\nonumber\\
&&\quad + \frac{1}{2\pi}\int_{-\pi}^{-w_0} e^{\log f(v+iw)}  e^{n(v+iw)}\,dw\nonumber\\
& =:& I_1 + I_2 +I_3.\label{eq:I1I2I3}
\eea
We will show that $I_1$ is the main term, and both $I_2$ and $I_3$ are lower order terms. 
\subsection{Estimate the main term $I_1$.} \label{subsec:mainterm}
Let $z=v+iw$ for $-w_0 \le w \le w_0$. Since $v=A(c_1,c_2)n^{-1/2}$ and $w_0 = n^{-3/4+\epsilon/3}$, by recalling \eqref{eq:lowerboundcA}, we have $\min(N,M)\cdot |z| \ge \min(c_1,c_2)\ge \bar c>0$ and $|w/v|\le |w_0/v|=o(1)$. Then we can apply Lemma \ref{lem:asymp} to obtain
 \beaa
 \log f(z) &=& \frac{1}{z} \left(\frac{\pi^2}{6} + \text{Li}_2(e^{-z(N+M)}) - \text{Li}_2(e^{-zN})-\text{Li}_2(e^{-zM})\right)\\
 && -\frac{1}{2}\log\left(\frac{1}{z}\right)-\frac{1}{2}\log \frac{(1-e^{-zN})(1-e^{-zM})}{1-e^{-z(N+M)}}-\log\sqrt{2\pi} + O(|z|).
 \eeaa
In particular,
  \bea\lbl{eq:real}
 \log f(v) &=& \frac{1}{v} \left(\frac{\pi^2}{6} + \text{Li}_2(e^{-v(N+M)}) - \text{Li}_2(e^{-vN})-\text{Li}_2(e^{-vM})\right)\nonumber\\
 && -\frac{1}{2}\log\left(\frac{1}{v}\right)-\frac{1}{2}\log \frac{(1-e^{-vN})(1-e^{-vM})}{1-e^{-v(N+M)}}-\log\sqrt{2\pi} + O(v).
 \eea
Now we compare the above two terms,
\bea\lbl{eq:complex}
&&\log f(v+iw)-\log f(v) \nonumber\\
&=& \frac{\pi^2}{6}\left(\frac{1}{z}-\frac{1}{v} \right) + \left(\frac{\text{Li}_2(e^{-z(N+M)})}{z} -\frac{\text{Li}_2(e^{-v(N+M)})}{v}\right) - \left(\frac{\text{Li}_2(e^{-zN})}{z} -\frac{\text{Li}_2(e^{-vN})}{v}\right) \nonumber\\
&&- \left(\frac{\text{Li}_2(e^{-zM})}{z} -\frac{\text{Li}_2(e^{-vM})}{v}\right) -\frac{1}{2}\left( \log\Big(\frac{1}{z}\Big)-\log\Big(\frac{1}{v}\Big) \right)+\frac{1}{2}\log\frac{1-e^{-z(N+M)}}{1-e^{-v(N+M)}}\nonumber \\
&& - \frac{1}{2}\log\frac{1-e^{-zN}}{1-e^{-vN}}- \frac{1}{2}\log\frac{1-e^{-zM}}{1-e^{-vM}}+ O(|z|)+ O(v)\nonumber\\
&=:& J_1 + J_2- J_3 -J_4 - H + K_1 - K_2 -K_3 + O(v).\label{eq:JHK}
\eea
We are going to estimate each term on the right hand side of \eqref{eq:JHK}.

Let us calculate $J_1$ first. Since
\beaa
\frac{1}{z}-\frac{1}{v} &=& \frac{1}{v+iw} - \frac{1}{v}=\frac{1}{v}\left( \frac{1}{1+i\frac{w}{v}} -1\right)\\
&=&-i\frac{w}{v^2} -\frac{w^2}{v^3} +O\left(\frac{w_0^3}{v^4}\right),
\eeaa
we obtain
\bea\lbl{eq:1z}
\frac{1}{z}=\frac{1}{v}-i\frac{w}{v^2} -\frac{w^2}{v^3} +O\left(\frac{w_0^3}{v^4}\right),
\eea
and thus,
\beaa
J_1 = \frac{\pi^2}{6}\left(-i\frac{w}{v^2} -\frac{w^2}{v^3}  \right)+O\left(\frac{w_0^3}{v^4}\right).
\eeaa

Now let us calculate $J_3$ in \eqref{eq:JHK}. Notice that $J_2$ and $J_4$ are in the same form as $J_3$, so they can be estimated in the same way. Using the property of Spence's function $\frac{\partial}{\partial u} \text{Li}_2(z) = -\frac{1}{z} \log(1-z)$ for $z \in \mathbb{C}\setminus [1,\infty)$, one has
\beaa
&&\frac{\partial}{\partial u} \text{Li}_2(e^{-u}) = e^{u}\log(1-e^{-u}) e^{-u} = \log(1-e^{-u});\\
&& \frac{\partial^2}{\partial u^2} \text{Li}_2(e^{-u}) = \frac{e^{-u}}{1-e^{-u}} = \frac{1}{e^u-1}.
\eeaa
Therefore, with noting that $N\cdot v = c_1$,
\beaa
\text{Li}_2(e^{-zN}) - \text{Li}_2(e^{-vN}) = iwN\log(1-e^{-vN}) - \frac{1}{2} \frac{1}{e^{vN}-1} w^2N^2 + O\left(\frac{w_0^3}{v^3}\right),
\eeaa
where in the last term $O(\cdot)$ we used that $c_1\ge \bar c>0$ independent of $n$.
Thus 
\beaa
&&J_3=\frac{\text{Li}_2(e^{-zN})}{z} -\frac{\text{Li}_2(e^{-vN})}{v} \\
&=& \frac{1}{z} \left( \text{Li}_2(e^{-zN}) -\text{Li}_2(e^{-vN})\right) + \left( \frac{1}{z} - \frac{1}{v}\right)  \text{Li}_2(e^{-vN})\\
&=&\left( \frac{1}{v}-i\frac{w}{v^2} -\frac{w^2}{v^3} +O\Big(\frac{w_0^3}{v^4}\Big) \right)\cdot \left(iwN\log(1-e^{-vN}) - \frac{1}{2} \frac{1}{e^{vN}-1} w^2N^2 + O\Big(\frac{w_0^3}{v^3}\Big) \right)\\
&&\quad + \left(-i\frac{w}{v^2} -\frac{w^2}{v^3} +O\Big(\frac{w_0^3}{v^4}\Big)\right)\text{Li}_2(e^{-vN}).
\eeaa
Since by our assumption $N \cdot v = c_1$, $v=A(c_1,c_2)\cdot n^{-1/2}$ and $w_0/v = o(1)$, it follows that 
\beaa
J_3 &=& i\frac{wc_1}{v^2}\log(1-e^{-c_1}) - \frac{w^2}{2v^3} \frac{c_1^2}{e^{c_1}-1} + \frac{w^2c_1}{v^3} \log(1-e^{-c_1})-i\frac{w}{v^2} \text{Li}_2(e^{-c_1})\\
&&\quad -\frac{w^2}{v^3}\text{Li}_2(e^{-c_1}) + O\left(\frac{w_0^3}{v^4}\right)\\
&=& - iwn \cdot \frac{-c_1 \log(1-e^{-c_1}) +\text{Li}_2(e^{-c_1})}{A(c_1,c_2)^2} \\
&&\quad-w^2\cdot \frac{n^{3/2}}{A(c_1,c_2)^3}\left(-c_1\log(1-e^{-c_1}) +\text{Li}_2(e^{-c_1}) + \frac{1}{2}\frac{c_1^2}{e^{c_1}-1} \right)+ O\left(\frac{w_0^3}{v^4}\right).
\eeaa
We estimate the other two terms $J_2$ and $J_4$ in  the same way.
Therefore
\beaa
&&J_1 + J_2- J_3 -J_4 \nonumber\\
&=& -iwn\cdot\frac{1}{A(c_1,c_2)^2}\Big(\frac{\pi^2}{6} +  \text{Li}_2(e^{-c_1-c_2})-\text{Li}_2(e^{-c_1})-\text{Li}_2(e^{-c_2})\nonumber\\
&&\quad -(c_1+c_2)\log(1-e^{-c_1-c_2}) + c_1 \log(1-e^{-c_1}) + c_2\log(1-e^{-c_2})\Big)\nonumber\\
&&\quad -w^2\cdot \frac{n^{3/2}}{A(c_1,c_2)^3}\Big(\frac{\pi^2}{6} +  \text{Li}_2(e^{-c_1-c_2})-\text{Li}_2(e^{-c_1})-\text{Li}_2(e^{-c_2})\nonumber\\
&&\quad -(c_1+c_2)\log(1-e^{-c_1-c_2}) + c_1 \log(1-e^{-c_1}) + c_2\log(1-e^{-c_2})\nonumber\\
&&\quad + \frac{1}{2}\frac{(c_1+c_2)^2}{e^{c_1+c_2}-1} - \frac{1}{2}\frac{c_1^2}{e^{c_1}-1} - \frac{1}{2}\frac{c_2^2}{e^{c_2}-1}\Big) + O\left(\frac{w_0^3}{v^4}\right).
\eeaa
By recalling \eqref{eq:A2}, we eventually get
\bea\lbl{eq:term1}
&&J_1 + J_2- J_3 -J_4 \nonumber\\
&=& - iwn - w^2\cdot \frac{n^{3/2}}{A(c_1,c_2)^3}\left(A(c_1,c_2)^2 + \frac{1}{2}\frac{(c_1+c_2)^2}{e^{c_1+c_2}-1} - \frac{1}{2}\frac{c_1^2}{e^{c_1}-1} -\frac{1}{2} \frac{c_2^2}{e^{c_2}-1}\right) \nonumber\\
&&\quad+ O\left(\frac{w_0^3}{v^4}\right).
\eea

 Now let us estimate $H$ in \eqref{eq:JHK}. 
 Since for $n$ sufficiently large, $|w/v| <1/2$, we have that
\beaa
\left|\log\Big(\frac{1}{z}\Big)-\log\Big(\frac{1}{v}\Big)\right| = \left|\log \Big(1+ i \frac{w}{v}\Big)\right| \le \sum_{k=1}^{\infty} \frac{1}{k}\left|\frac{w}{v}\right|^k \le \frac{|w/v|}{1-|w/v|} =O\left(\frac{w_0}{v}\right),
\eeaa
and thus
\bea\lbl{eq:term2}
H=O\left(\frac{w_0}{v}\right).
\eea

Now let us estimate $K_1,K_2$ and $K_3$ in \eqref{eq:JHK}. These three quantities are in the same form, so we just show the details for $K_2$.
Since
\beaa
&&\frac{\partial}{\partial u} \log(1-e^{-u}) =\frac{e^{-u}}{1-e^{-u}} = \frac{1}{e^u-1};\\
&& \frac{\partial^2}{\partial u^2} \log(1-e^{-u}) = -\frac{e^u}{(e^u-1)^2},
\eeaa
we obtain
\beaa
2K_2=\log\frac{1-e^{-zN}}{1-e^{-vN}} &=& \log(1-e^{-vN-iwN}) - \log(1-e^{-vN}) \\
&=&\frac{iwN}{e^{vN}-1} + \frac{1}{2} \frac{e^{vN}}{(e^{vN}-1)^2} w^2 N^2 + O\left(\frac{w_0^3}{v^3}\right).
\eeaa
Since $N \cdot v = c_1$, $v=A(c_1,c_2)\cdot n^{-1/2}$ and $w_0/v = o(1)$,
\beaa
K_2=O\left(\frac{w_0}{v}\right).
\eeaa
Similarly, $K_i= O(\frac{w_0}{v})$ for $i=1,3$ and thus,
\bea\lbl{eq:term3}
K_1-K_2 -K_3=O\left(\frac{w_0}{v}\right).
\eea

Plugging \eqref{eq:term1}, \eqref{eq:term2} and \eqref{eq:term3} into \eqref{eq:complex}, we arrive at
\bea
&&\log f(v+iw)-\log f(v) \nonumber\\
&=& - iwn - w^2\cdot \frac{n^{3/2}}{A(c_1,c_2)^3}\left(A(c_1,c_2)^2 + \frac{1}{2}\frac{(c_1+c_2)^2}{e^{c_1+c_2}-1} - \frac{1}{2}\frac{c_1^2}{e^{c_1}-1} -\frac{1}{2} \frac{c_2^2}{e^{c_2}-1}\right) \nonumber\\
&&\quad+ O\left(\frac{w_0^3}{v^4}\right) + O(v) + O\left(\frac{w_0}{v}\right)\nonumber\\
&=&- iwn - \frac{1}{2}w^2 n^{3/2}S(c_1,c_2) + O(n^{-1/4+\epsilon}),\label{eq:firsttermaux}
\eea
where for convenience we denoted
\beaa
S(c_1,c_2)=\frac{1}{A(c_1,c_2)^3}\left(2A(c_1,c_2)^2 + \frac{(c_1+c_2)^2}{e^{c_1+c_2}-1} - \frac{c_1^2}{e^{c_1}-1} - \frac{c_2^2}{e^{c_2}-1}\right).
\eeaa
In obtaining the term of $O(n^{-1/4+\epsilon})$ in \eqref{eq:firsttermaux}, we used that $A(c_1,c_2)$ is bounded from above and below by two universal positive constants (see \eqref{eq:Aboundabove} and \eqref{eq:lowerboundcA}). 

We now show that $S(c_1,c_2)$ is bounded from above and from below by two universal positive constants. Recall that 
\beaa
A(c_1,c_2)^2=\int_{0}^{c_1} \frac{x}{e^x-1}\,dx -\int_{c_2}^{c_1+c_2} \frac{x}{e^x-1}\,dx.
\eeaa
By integration by parts (noticing that we assumed $\alpha\le\beta$, so that $c_1\le c_2$), we have
\begin{equation}\label{eq:positive}
\begin{split}
L(c_1,c_2)&:=2A(c_1,c_2)^2 + \frac{(c_1+c_2)^2}{e^{c_1+c_2}-1} - \frac{c_1^2}{e^{c_1}-1} - \frac{c_2^2}{e^{c_2}-1} \\
&=\int_{0}^{c_1} \frac{x^2 e^x}{(e^x-1)^2}\,dx -\int_{c_2}^{c_1+c_2} \frac{x^2 e^x}{(e^x-1)^2}\,dx\\
&\ge \int_{0}^{c_1} \frac{x^2 e^x}{(e^x-1)^2}\,dx -\int_{c_1}^{2c_1} \frac{x^2 e^x}{(e^x-1)^2}\,dx=:\widetilde B(c_1).
\end{split}
\end{equation}
as the function ${x^2 e^x}/{(e^x-1)^2}$ is strictly decreasing for $x\ge0$ (which can be elementarily verified by showing that its derivative is negative). 
Since
\[
\frac{d}{dc }\widetilde B(c)=\frac{2c^2e^c}{(e^c+1)^2}>0,
\]
$\widetilde B(c)$ is increasing in $c$. Thus, continuing from \eqref{eq:positive} and using $\min(c_1,c_2)\ge \bar c$ in \eqref{eq:lowerboundcA}, we obtain
\begin{equation}\label{eq:Lboundedbelow}
0<\widetilde B(\bar c)\le \widetilde B(c_1)\le L(c_1,c_2)\le \int_{0}^{\infty} \frac{x^2 e^x}{(e^x-1)^2}\,dx<\infty.
\end{equation}
Using \eqref{eq:Aboundabove} and \eqref{eq:lowerboundcA}, we conclude that $S(c_1,c_2)=L(c_1,c_2)A(c_1,c_2)^{-3}$ is bounded from above and from below by two universal positive constants.

Using \eqref{eq:firsttermaux}, we can estimate the  term
\begin{align}
I_1 &= \frac{1}{2\pi} \int_{-w_0}^{w_0} e^{\log f(v+iw)} e^{n(v+iw)}\,dw\nonumber\\
&= \frac{1}{2\pi} \int_{-w_0}^{w_0} e^{\log f(v)-iwn-\frac{1}{2}w^2n^{3/2}S(c_1,c_2) + O(n^{-1/4+\epsilon})} e^{n(v+iw)}\,dw\nonumber\\
&=\exp(\log f(v) + nv +O(n^{-1/4+\epsilon})) \cdot \frac{1}{2\pi}\int_{-w_0}^{w_0} e^{-\frac{1}{2}w^2 n^{3/2}S(c_1,c_2)}\,dw.\label{eq:I1estimate0}
\end{align}
For the first term, since we have from \eqref{eq:real} that
\beaa
\log f(v) &=& \frac{n^{1/2}}{A(c_1,c_2)} \left(\frac{\pi^2}{6} + \text{Li}_2(e^{-c_1-c_2}) - \text{Li}_2(e^{-c_1})-\text{Li}_2(e^{-c_2})\right)\nonumber\\
 && -\frac{1}{2}\log\left(\frac{\sqrt{n}}{A(c_1,c_2)}\right)-\frac{1}{2}\log \frac{(1-e^{-c_1})(1-e^{-c_2})}{1-e^{-c_1-c_2}}-\log\sqrt{2\pi} + O(v),
\eeaa
it follows that
\begin{align}
& \exp(\log f(v) + nv +O(n^{-1/4+\epsilon})) \nonumber\\
&=\frac{1}{\sqrt{2\pi}} \frac{\sqrt{A(c_1,c_2)}}{n^{1/4}} \sqrt{\frac{1-e^{-(c_1+c_2)}}{(1-e^{-c_1})(1-e^{-c_2})}}\cdot e^{\sqrt{n}K(c_1,c_2)} \Big(1+ O(n^{-1/4+\epsilon})\Big),\label{eq:I1estimate1}
\end{align}
where
\beaa
K(c_1,c_2) = A(c_1,c_2) + \frac{1}{A(c_1,c_2)}\left(\frac{\pi^2}{6} + \text{Li}_2(e^{-(c_1+c_2)}) - \text{Li}_2(e^{-c_1})-\text{Li}_2(e^{-c_2})\right).
\eeaa
For the other term, by setting $t=\sqrt{S(c_1,c_2)}n^{3/4}\cdot w$, we get
\bea
&&\int_{-w_0}^{w_0} e^{-\frac{1}{2}w^2 n^{3/2}S(c_1,c_2)}\,dw=\frac{1}{\sqrt{S(c_1,c_2)}n^{3/4}}\int_{-\sqrt{S(c_1,c_2)}n^{\epsilon/3}}^{\sqrt{S(c_1,c_2)}n^{\epsilon/3}} e^{-t^2/2}\,dt\nonumber\\
&&=\frac{1}{\sqrt{S(c_1,c_2)}n^{3/4}}\left(\int_{-\infty}^{\infty} e^{-t^2/2}\,dt-2\int_{\sqrt{S(c_1,c_2)}n^{\epsilon/3}}^{\infty} e^{-t^2/2}\,dt \right)\nonumber\\
&&=\frac{1}{\sqrt{S(c_1,c_2)}n^{3/4}}\left(\sqrt{2\pi} + O\Big(\frac{2}{\sqrt{S(c_1,c_2)}n^{\epsilon/3}} e^{-\frac{S(c_1,c_2)}{2} n^{2\epsilon/3}}  \Big)\right)\nonumber\\
&&=\frac{\sqrt{2\pi}}{\sqrt{S(c_1,c_2)}n^{3/4}}\left(1 + O(n^{-1/4+\epsilon})\right),\label{eq:I1estimate2}
\eea
where in the last equality we used the fact that $S(c_1,c_2)$ bounded from above and from below by two universal positive constants.

Finally, putting \eqref{eq:I1estimate0}, \eqref{eq:I1estimate1} and \eqref{eq:I1estimate2} together, we have
\bea\lbl{eq:mainterm}
I_1 &=& \frac{1}{2\pi} \frac{\sqrt{A(c_1,c_2)}}{\sqrt{S(c_1,c_2)}} \sqrt{\frac{1-e^{-(c_1+c_2)}}{(1-e^{-c_1})(1-e^{-c_2})}}\cdot \frac{e^{\sqrt{n}K(c_1,c_2)}}{n} (1+ O(n^{-1/4+\epsilon})) \nonumber\\
&=&\frac{A(c_1,c_2)^2 }{ 2\pi\sqrt{L(c_1,c_2)}}  \sqrt{\frac{1-e^{-(c_1+c_2)}}{(1-e^{-c_1})(1-e^{-c_2})}}\cdot\frac{e^{\sqrt{n}K(c_1,c_2)}}{n}(1+ O(n^{-1/4+\epsilon})),
\eea
where
\beaa
&&K(c_1,c_2) = A(c_1,c_2) + \frac{1}{A(c_1,c_2)}\left(\frac{\pi^2}{6} + \text{Li}_2(e^{-(c_1+c_2)}) - \text{Li}_2(e^{-c_1})-\text{Li}_2(e^{-c_2})\right);\\
&&L(c_1,c_2) =2A(c_1,c_2)^2 + \frac{(c_1+c_2)^2}{e^{c_1+c_2}-1} -  \frac{c_1^2}{e^{c_1}-1}-\frac{c_2^2}{e^{c_2}-1}.
\eeaa

We know from \eqref{eq:Aboundabove}, \eqref{eq:lowerboundcA} and \eqref{eq:Lboundedbelow} that $A(c_1,c_2)$ and $L(c_1,c_2)$ are bounded from above and from below by two universal positive constants.
 Since $\min(c_1,c_2)\ge\bar c>0$, we have
\begin{equation}\label{eq:coefficientexponential}
\sqrt{1-e^{-2\bar c}}\le \sqrt{\frac{1-e^{-(c_1+c_2)}}{(1-e^{-c_1})(1-e^{-c_2})}}\le \frac{1}{1-e^{-\bar c}}.
\end{equation}
Using \eqref{eq:A3} and $\text{Li}_2(1)=\frac{\pi^2}{6}$, we have 
\begin{equation}\label{eq:Kbound}
2A(c_1,c_2)\le K(c_1,c_2)\le A(c_1,c_2)+ \frac{\pi^2}{3A(c_1,c_2)}.
\end{equation}
This implies that $K(c_1,c_2)$ is also bounded from above and from below by two universal positive constants.

Hence, all of $A(c_1,c_2)$, $L(_1,c_2)$, $K(c_2,c_2)$ and $\sqrt{\frac{1-e^{-(c_1+c_2)}}{(1-e^{-c_1})(1-e^{-c_2})}}$ in \eqref{eq:mainterm} are bounded from above and from below by two universal positive constants.  To prove that $I_2$ and $I_3$ in \eqref{eq:I1I2I3} are lower order terms compared to $I_1$, it suffices to show that
\[
|I_2|+|I_3|=o\left(\frac{\exp\{\sqrt{n}K(c_1,c_2)\}}{n^{\frac{5}{4}-\varepsilon}}\right).
\]
We will prove the above estimate in the next subsection.

\subsection{Estimate the remainder terms $I_2$ and $I_3$.} \label{subsec:error}
It is enough to estimate $I_2$ in \eqref{eq:I1I2I3}. The proof for $I_3$ is identical. Let $w_1=C\cdot v$ where $C$ is a universally large constant to be fixed in {\it Step 2} below (one will see later that $C=20\pi$ would suffice). We further split $I_2$ into two parts.
\begin{align}
I_2&=\frac{1}{2\pi}\int_{w_0}^{\pi} f(v+iw) e^{n(v+iw)}\,dw \nonumber\\
&= \frac{1}{2\pi}\int_{w_0}^{w_1} f(v+iw) e^{n(v+iw)}\,dw+\frac{1}{2\pi}\int_{w_1}^{\pi} f(v+iw) e^{n(v+iw)}\,dw\nonumber\\
&=:I_2' + I_2''.\label{eq:splitI2}
\end{align}

{\it Step 1. Estimate the term $I_2'$.} Note that for $z=v+iw$ with $w_0\le w\le w_1$, the hypothesis of Lemma \ref{lem:asymp} is satisfied. Thus
\beaa
 \log f(z) &=& \frac{1}{z} \left(\frac{\pi^2}{6} + \text{Li}_2(e^{-z(N+M)}) - \text{Li}_2(e^{-zN})-\text{Li}_2(e^{-zM})\right)\\
 && -\frac{1}{2}\log\left(\frac{1}{z}\right)-\frac{1}{2}\log \frac{(1-e^{-zN})(1-e^{-zM})}{1-e^{-z(N+M)}} + O(1)
\eeaa
and
\beaa
|I_2'|&=& \left|\frac{1}{2\pi}\int_{w_0}^{w_1} e^{\log f(v+iw) +n(v+iw)}\,dw\right|\\
&=&\Big| e^{O(1)}\int_{w_0}^{w_1} \sqrt{z} \cdot \sqrt{\frac{1-e^{-z(N+M)}}{(1-e^{-zN})(1-e^{-zM})}} \\
&&\quad \cdot\exp\left\{\frac{1}{z} \Big(\frac{\pi^2}{6} + \text{Li}_2(e^{-z(N+M)}) - \text{Li}_2(e^{-zN})-\text{Li}_2(e^{-zM})\Big) + nv +inw\right\} \,dw\Big|\\
&\lesssim& \int_{w_0}^{w_1} \sqrt{|z|} \cdot \sqrt{\frac{|1-e^{-z(N+M)}|}{|1-e^{-zN}|\cdot|1-e^{-zM}|}} \\
&&\quad \cdot \exp\left\{ \text{Re}\frac{1}{z} \Big(\frac{\pi^2}{6} + \text{Li}_2(e^{-z(N+M)}) - \text{Li}_2(e^{-zN})-\text{Li}_2(e^{-zM}) \Big)+ nv \right\} \,dw.
\eeaa
Since $|z| =O(v)$, and 
\beaa
2\ge |1-e^{-zN}|=|1-e^{-vN-iwN}| \ge 1- |e^{-vN-iwN}| = 1- e^{-vN} = 1-e^{-c_1},
\eeaa
we conclude that (recalling $v=A(c_1,c_2) n^{-1/2}$),
\begin{align}
|I_2'| &\lesssim\sqrt{v}\int_{w_0}^{w_1} \exp\{ \text{Re}\frac{1}{z} \big(\frac{\pi^2}{6} + \text{Li}_2(e^{-z(N+M)}) - \text{Li}_2(e^{-zN})-\text{Li}_2(e^{-zM})\big) +nv\} \,dw\nonumber\\
&=\sqrt{v}\cdot e^{\sqrt{n}K(c_1,c_2)} \int_{w_0}^{w_1}\exp\Big\{ \text{Re}\frac{1}{z} \Big(\frac{\pi^2}{6} + \text{Li}_2(e^{-z(N+M)}) - \text{Li}_2(e^{-zN})-\text{Li}_2(e^{-zM})\Big) \nonumber\\
&~\quad \quad\quad\quad\quad\quad \quad\quad -\frac{1}{v} \Big(\frac{\pi^2}{6} + \text{Li}_2(e^{-v(N+M)}) - \text{Li}_2(e^{-vN})-\text{Li}_2(e^{-vM}) \Big)\Big\} \,dw.\label{eq:I2'}
\end{align}
Let us denote
\bea\lbl{eq:difference}
&&\mathcal{D}:=\frac{1}{z} \left(\frac{\pi^2}{6} + \text{Li}_2(e^{-z(N+M)}) - \text{Li}_2(e^{-zN})-\text{Li}_2(e^{-zM})\right) \nonumber\\
&&\quad\quad\quad-\frac{1}{v} \left(\frac{\pi^2}{6} + \text{Li}_2(e^{-v(N+M)}) - \text{Li}_2(e^{-vN})-\text{Li}_2(e^{-vM}) \right).
\eea
Next we are going to show that 
\begin{align}\label{eq:estimateII}
\text{Re}(\mathcal{D}) \le -\frac{w^2}{v^3} \cdot T 
\end{align}
for some universal constant $T >0$, where $\text{Re}(\mathcal{D})$ is the real part of $\mathcal{D}$.

We first get
\beaa
\text{Re}\left(\frac{\text{Li}_2(e^{-c_1-iwN})}{z} \right)&=&\text{Re}\Big(\frac{v-iw}{v^2 +w^2}\text{Li}_2(e^{-c_1-iwN}) \Big)\\
&=&\frac{1}{v^2+w^2} \Big\{ v\text{Re}(\text{Li}_2(e^{-c_1-iwN})) + w \text{Im}(\text{Li}_2(e^{-c_1-iwN}))\Big\}\\
&=&\frac{v}{v^2+w^2} \Big\{ \text{Re}(\text{Li}_2(e^{-c_1-ic_1\frac{w}{v}})) + \frac{w}{v}\cdot \text{Im}(\text{Li}_2(e^{-c_1-ic_1\frac{w}{v}}))\Big\}\\
&=&\frac{1}{v}\cdot \frac{1}{1+(w/v)^2} \Big\{ \text{Re}(\text{Li}_2(e^{-c_1-ic_1\frac{w}{v}})) + \frac{w}{v}\cdot \text{Im}(\text{Li}_2(e^{-c_1-ic_1\frac{w}{v}}))\Big\}.
\eeaa
Therefore,
\beaa
v\cdot \text{Re}(\mathcal{D}) &=& \frac{1}{1+(\frac{w}{v})^2}\left(\frac{\pi^2}{6} +\eta_{c_1+c_2}\Big(\frac{w}{v}\Big) - \eta_{c_1}\Big(\frac{w}{v}\Big) -\eta_{c_2}\Big(\frac{w}{v}\Big) \right)\\
&&\quad -\left(\frac{\pi^2}{6} +\eta_{c_1+c_2}(0) - \eta_{c_1}(0) -\eta_{c_2}(0) \right)\\
&=&H\left(\frac{w}{v}\right)-H(0),
\eeaa
where $\eta_c$ is defined in \eqref{eq:definitionf} and $H$ is defined in \eqref{eq:definitionG}. By Lemma \ref{lem:technical}, we have for $0<w/v\le C$ that
\[
v\cdot \text{Re}(\mathcal{D}) =H\left(\frac{w}{v}\right)-H(0)\le -\frac{\delta}{1+C^2} \frac{w^2}{v^2}=: - T\frac{w^2}{v^2}.
\]
This proves \eqref{eq:estimateII}. 



Therefore, continuing from \eqref{eq:I2'} with the help of \eqref{eq:estimateII}, we have that
\begin{align}
|I_2'|&\lesssim\sqrt{v}\cdot e^{\sqrt{n}K(c_1,c_2)} \int_{w_0}^{w_1}\exp\{-w^2T/v^3\}dw\nonumber\\
&\lesssim \sqrt{v}\cdot e^{\sqrt{n}K(c_1,c_2)} \frac{1}{w_0}\int_{w_0}^{w_1}w\exp\{-w^2T/v^3\}dw\nonumber\\
&\lesssim \frac{A(c_1,c_2)^{\frac 72}}{2Tn^{1+\frac{\epsilon}{3}}}\exp\left\{\sqrt{n}K(c_1,c_2)-\frac{Tn^{\frac{2\varepsilon}{3}}}{A(c_1,c_2)^3}\right\}\nonumber\\
&=o\left(\frac{\exp\{\sqrt{n}K(c_1,c_2)\}}{n^{\frac{5}{4}-\varepsilon}}\right).\label{eq:I2final1}
\end{align}
This finishes the estimate of $I_2'$, which is indeed a lower order term compared to $I_1$. 

{\it Step 2. Estimate the term $I_2''$.} 
Recall that
\beaa
 \log f(z)=\sum_{k=1}^{\infty} \frac{1}{k} \frac{(1-e^{-kz\cdot N})(1-e^{-kz\cdot M})}{e^{kz}-1}
\eeaa
and
\begin{equation}\label{eq:definitionofI2b}
I_2''=\frac{1}{2\pi}\int_{w_1}^{\pi} e^{\log f(v+iw) +n(v+iw)}\,dw,
\end{equation}
where $w_1=C\cdot v$ with $C$ being  a large constant (to be determined below). 

First, note that $|e^{k(v+iw)}-1|  \ge e^{kv}-1 \ge kv$. Secondly, we have
\beaa
|e^{z}-1| &=& |e^{v + iw}-1| = \sqrt{(e^{v}\cos w-1)^2 + e^{2v}\sin^2w}\\
&=&\sqrt{(e^{v}-1)^2 + 4 e^{v}\sin^2(w/2)},
\eeaa
and thus, 
\beaa
|e^{v+iw}-1| \ge 2\sin(w/2) \ge 2w/\pi
\eeaa 
which follows from the inequality that $\sin\theta \ge \frac{2}{\pi}\theta$ for $0\le \theta \le \pi/2$.

It follows that
\beaa
& &|\log f(v+iw)| \\
&=& \Big|\sum_{k=1}^{\infty} \frac{1}{k} \frac{(1-e^{-k(v+iw)N})(1-e^{-k(v+iw)M})}{e^{k(v+iw)-1}}\Big|\\
&\le & \frac{2\pi }{Cv}+\frac{1}{v}\sum_{k=2}^{\infty}\frac{(1+e^{-kc_1})(1+e^{-kc_2})}{k^2}\\
&\le & \frac{2\pi }{Cv}+ \frac{1}{v}\sum_{k=1}^{\infty}\frac{(1+e^{-kc_1})(1+e^{-kc_2})}{k^2} -\frac{1}{v}(1+e^{-c_1})(1+e^{-c_2})\\
&=&  \frac{1}{v}\left(\frac{2\pi }{C}+\frac{\pi^2}{6} + \text{Li}_2(e^{-(c_1+c_2)}) + \text{Li}_2(e^{-c_1})+\text{Li}_2(e^{-c_2}) -(1+e^{-c_1})(1+e^{-c_2})\right).
\eeaa
We claim that for $C$ sufficiently large, we have
\begin{align}\label{eq:conditionC}
&\frac{4\pi }{C}+\frac{\pi^2}{6} + \text{Li}_2(e^{-(c_1+c_2)}) + \text{Li}_2(e^{-c_1})+\text{Li}_2(e^{-c_2}) -(1+e^{-c_1})(1+e^{-c_2}) \nonumber\\
& \quad \le \frac{\pi^2}{6} + \text{Li}_2(e^{-(c_1+c_2)}) - \text{Li}_2(e^{-c_1})-\text{Li}_2(e^{-c_2}).
\end{align}
Suppose the above inequality \eqref{eq:conditionC} holds. Then continuing from \eqref{eq:definitionofI2b}, we have
\begin{align}
|I_2''|&\le \frac{1}{2\pi}\int_{w_1}^{\pi} e^{|\log f(v+iw)| +nv}\,dw\nonumber\\
&\le \frac{1}{2\pi}\int_{w_1}^{\pi} e^{\sqrt{n}(K(c_1,c_2)-2\pi/C)}\,dw=o\left(\frac{\exp\{\sqrt{n}K(c_1,c_2)\}}{n^{\frac{5}{4}-\varepsilon}}\right).\label{eq:I2final2}
\end{align}
Now we are left to prove \eqref{eq:conditionC}, which is equivalent to
\[
\frac{2\pi}{C}+ \text{Li}_2(e^{-c_1})+\text{Li}_2(e^{-c_2}) \le \frac{(1+e^{-c_1})(1+e^{-c_2})}{2}.
\]
Since we assume that $\min(\alpha,\beta)\ge 4$, we know from Lemma \ref{lem:changeofvariables} that $\min(c_1,c_2)\ge 13/5>2$. Hence,   $$\text{Li}_2(e^{-c_1})\le \frac{e^{-c_1}}{1-e^{-c_1}}=\frac{1}{e^{c_1}-1}\le \frac{1}{e^{2}-1}<\frac 15,$$
and similarly $\text{Li}_2(e^{-c_2})\le 1/5$. Hence,  we only need to choose $C$ large so that
\[
\frac{2\pi}{C}+\frac{2}{5}\le \frac 12.
\]
Therefore, $C= 20\pi$ would suffice.

From \eqref{eq:splitI2}, \eqref{eq:I2final1} and \eqref{eq:I2final2}, we obtained that
\begin{equation}\label{eq:I2final3}
|I_2|\le o\left(\frac{\exp\{\sqrt{n}K(c_1,c_2)\}}{n^{\frac{5}{4}-\varepsilon}}\right).
\end{equation}
Since $I_3$ is in the same form as $I_2$ defined in \eqref{eq:I1I2I3}, one can show similarly that 
\begin{equation}\label{eq:I3final3}
|I_3|\le o\left(\frac{\exp\{\sqrt{n}K(c_1,c_2)\}}{n^{\frac{5}{4}-\varepsilon}}\right).
\end{equation}

\subsection{Asymptotic formula}
Combining \eqref{eq:I1I2I3}, \eqref{eq:mainterm}, \eqref{eq:I2final3} and \eqref{eq:I3final3}, we obtain
\begin{align}\label{eq:another}
p_n(N,M)=\frac{A(c_1,c_2)^2 }{ 2\pi\sqrt{L(c_1,c_2)}}  \sqrt{\frac{1-e^{-(c_1+c_2)}}{(1-e^{-c_1})(1-e^{-c_2})}}\cdot\frac{e^{\sqrt{n}K(c_1,c_2)}}{n}\left(1+ O(n^{-1/4+\epsilon})\right),
\end{align}
where
\beaa
&&K(c_1,c_2) = A(c_1,c_2) + \frac{1}{A(c_1,c_2)}\left(\frac{\pi^2}{6} + \text{Li}_2(e^{-(c_1+c_2)}) - \text{Li}_2(e^{-c_1})-\text{Li}_2(e^{-c_2})\right);\\
&&L(c_1,c_2) =2A(c_1,c_2)^2 + \frac{(c_1+c_2)^2}{e^{c_1+c_2}-1} -  \frac{c_1^2}{e^{c_1}-1}-\frac{c_2^2}{e^{c_2}-1}.
\eeaa
By recalling \eqref{eq:A} and \eqref{eq:positive}, we have the explicit integral representation for $A(c_1,c_2)$ and $L(c_1,c_2)$:
\begin{align*}
A(c_1,c_2)&=\sqrt{\int_{0}^{c_1} \frac{x}{e^x-1}\,dx -\int_{c_2}^{c_1+c_2} \frac{x}{e^x-1}\,dx},\\
L(c_1,c_2)&=\int_{0}^{c_1} \frac{x^2 e^x}{(e^x-1)^2} \,dx -\int_{c_2}^{c_1+c_2} \frac{x^2 e^x}{(e^x-1)^2} \,dx.
\end{align*}
Therefore, all the terms in  the equation  \eqref{eq:another} are explicit.

The last step is to express the asymptotic formula in terms of $\alpha, \beta$ instead of $c_1, c_2$. Recall that $c_1 = \alpha A(c_1, c_2)$ and $c_2 = \beta A(c_1, c_2)$. We denote $g(\alpha,\beta) = A(c_1, c_2)>0$ and $g(\alpha,\beta)$ satisfies
\beaa
g^2(\alpha,\beta) = \int_{0}^{\alpha g(\alpha,\beta)} \frac{x}{e^x-1} \,dx - \int_{\beta g(\alpha,\beta)}^{(\alpha+\beta)g(\alpha,\beta)} \frac{x}{e^x-1} \,dx
\eeaa
or equivalently,
\beaa
g^2(\alpha,\beta) = \text{Li}_2(1-e^{-\alpha g(\alpha,\beta)}) + \text{Li}_2(1-e^{-\beta g(\alpha,\beta)}) - \text{Li}_2 (1-e^{-(\alpha+\beta)g(\alpha,\beta)}).
\eeaa
The uniqueness and existence of $g(\alpha,\beta)>0$ is guaranteed by Lemma \ref{lem:changeofvariables}. Then by recalling \eqref{eq:another}, we have
\begin{align}\label{eq:pNMfinal}
&p_n(N,M)\nonumber\\
&= \frac{g(\alpha,\beta)^2 }{ 2\pi\sqrt{L(\alpha,\beta)}}  \sqrt{\frac{1-e^{-(\alpha+\beta)g(\alpha,\beta)}}{(1-e^{-\alpha g(\alpha,\beta)})(1-e^{-\beta g(\alpha,\beta)})}}\cdot\frac{e^{\sqrt{n}K(\alpha,\beta)}}{n}(1+ O(n^{-1/4+\epsilon})),
\end{align}
where
\begin{align*}
K(\alpha,\beta)= g(\alpha,\beta) + \frac{1}{g(\alpha,\beta)}\left(\frac{\pi^2}{6} + \text{Li}_2(e^{-(\alpha+\beta)g(\alpha,\beta)}) - \text{Li}_2(e^{-\alpha g(\alpha,\beta)}) -\text{Li}_2 (e^{-\beta g(\alpha,\beta)})\right)
\end{align*}
and
\beaa
L(\alpha,\beta) = 2 g^2(\alpha,\beta) + \frac{(\alpha+\beta)^2 g^2(\alpha,\beta)}{e^{(\alpha+\beta)g(\alpha,\beta)}-1} -\frac{\alpha^2 g^2(\alpha,\beta)}{e^{\alpha g(\alpha,\beta)}-1} -\frac{\beta^2 g^2(\alpha,\beta)}{e^{\alpha g(\alpha,\beta)}-1}.
\eeaa

Since $\min(\alpha,\beta)\ge 4$, it follows from Lemma \ref{lem:changeofvariables} that $\min(c_1,c_2)\ge 13/5$. By \eqref{eq:Aboundabove}, \eqref{eq:lowerboundcA}, \eqref{eq:Lboundedbelow}, \eqref{eq:coefficientexponential} and \eqref{eq:Kbound}, we know all of $A(c_1,c_2)$, $L(_1,c_2)$, $K(c_2,c_2)$ and $\sqrt{\frac{1-e^{-(c_1+c_2)}}{(1-e^{-c_1})(1-e^{-c_2})}}$ in \eqref{eq:another} are bounded from above and from below by two universal positive constants. Hence, all the functions $g(\alpha,\beta)$, $L(\alpha,\beta)$, $K(\alpha,\beta)$ and  $ \sqrt{\frac{1-e^{-(\alpha+\beta)g(\alpha,\beta)}}{(1-e^{-\alpha g(\alpha,\beta)})(1-e^{-\beta g(\alpha,\beta)})}}$ in \eqref{eq:pNMfinal} are bounded from above and from below by two universal positive constants. The error term $O(n^{-1/4+\epsilon})$ is also uniform in $\alpha,\beta\in[4,\infty]$. 
\end{proof}



\end{document}